\documentclass[reqno,oneside,12pt]{amsart}

\usepackage[english]{babel}
\usepackage[utf8]{inputenc}
\usepackage{amsfonts,amsmath,amsthm}
\usepackage{stmaryrd}
\usepackage{thmtools,thm-restate}
\usepackage[T1]{fontenc}
\usepackage{enumitem}
\usepackage{varioref}
\usepackage[all]{xy}
\usepackage{units}
\usepackage{hyperref}
\usepackage{graphicx}
\usepackage{amssymb}
\usepackage{mathabx}
\usepackage{times,mathptm, mathtools}
\usepackage{etoolbox}
\usepackage{tikz-cd}
\usepackage{mathrsfs}
\usepackage{mathdots}
\usepackage[left=2.5cm, right=2.5cm]{geometry}

\docsvlist{A,B,C,D,E,F,G,H,I,J,K,L,M,N,O,P,Q,R,S,T,U,V,W,X,Y,Z}

\docsvlist{A,B,C,D,E,F,G,H,I,J,K,L,M,N,O,P,Q,R,S,T,U,V,W,X,Y,Z}

\docsvlist{A,B,C,D,E,F,G,H,I,J,K,L,M,N,O,P,Q,R,S,T,U,V,W,X,Y,Z}

\docsvlist{A,B,C,D,E,F,G,H,I,J,K,L,M,N,O,P,Q,R,S,T,U,V,W,X,Y,Z}

\docsvlist{A,B,C,D,E,F,G,H,I,J,K,L,M,N,O,P,Q,R,S,T,U,V,W,X,Y,Z}

\newcommand{\R}{\mathbf{R}}
\newcommand{\C}{\mathbf{C}}
\newcommand{\Q}{\mathbf{Q}}
\newcommand{\Z}{\mathbf{Z}}

\newcommand{\A}{\mathbf{A}}
\newcommand{\F}{\mathbf{F}}
\newcommand{\K}{\mathbf{K}}
\renewcommand{\P}{\mathbf{P}}

\newcommand{\m}{\mathfrak{m}}

\newcommand{\G}{\mathbb G}

\newcommand{\HH}{\mathbb H}
\newcommand{\CS}{\mathbb S}

\DeclareMathOperator{\Pic}{Pic}
\DeclareMathOperator{\Div}{Div}

\DeclareMathOperator{\ord}{ord}
\newcommand{\OO}{\mathcal O}
\DeclareMathOperator{\id}{id}
\DeclareMathOperator{\Tr}{Tr}

\DeclareMathOperator{\supp}{Supp}

\DeclareMathOperator{\NS}{NS}

\DeclareMathOperator{\Ind}{Ind}

\DeclareMathOperator{\spec}{Spec}

\DeclareMathOperator{\Aut}{Aut}

\DeclareMathOperator{\GL}{GL}

\newcommand{\DivInf}{\Div_\infty}
\DeclareMathOperator{\car}{char}

\DeclareMathOperator{\tr.deg}{tr.deg}

\DeclareMathOperator{\Weil}{Weil}
\DeclareMathOperator{\Cartier}{Cartier}
\DeclareMathOperator{\Cinf}{\Cartier_\infty (X_0)}
\DeclareMathOperator{\Winf}{\Weil_\infty (X_0)}
\DeclareMathOperator{\Bir}{Bir}
\DeclareMathOperator{\Stab}{Stab}
\newcommand{\Torus}{\G_m^2}

\newtheorem{thm}{Theorem}[section]
\newtheorem{thm*}{Theorem}

\newtheorem{prop}[thm]{Proposition}
\newtheorem{cor}[thm]{Corollary}
\newtheorem{lemme}[thm]{Lemma}
\newtheorem{conj}[thm]{Conjecture}

\theoremstyle{definition}

\newtheorem{ex}[thm]{Example}
\newtheorem{rmq}[thm]{Remark}

\AtBeginEnvironment{dfn}{%
  \setlist[enumerate]{label={(\roman*)}}
}

\AtBeginEnvironment{thm}{%
  \setlist[enumerate,1]{label={(\arabic*)}}
}

\AtBeginEnvironment{prop}{%
  \setlist[enumerate,1]{label={(\arabic*)}}
}

\AtBeginEnvironment{lemme}{%
  \setlist[enumerate,1]{label={(\arabic*)}}
}

\begin{document}
\title{Intersection of orbits of loxodromic automorphisms of affine surfaces}
\author{Marc Abboud}
\address{Marc Abboud, Institut de mathématiques, Université de Neuchâtel
\\ Rue Emile-Argand 11 CH-2000 Neuchâtel}
\email{marc.abboud@normalesup.org}
\thanks{The author acknowledge support by the Swiss National Science Foundation Grant “Birational transformations of higher dimensional varieties” 200020-214999.}
\maketitle

\begin{abstract}
We show the following result: If $X_0$ is an affine surface over a field $K$ and $f,g$ are two loxodromic
automorphisms with an orbit meeting infinitely many times, then $f$ and $g$ must share a common iterate. The proof uses
the preliminary work of the author in \cite{abboudDynamicsEndomorphismsAffine2023} on the dynamics of endomorphisms of
affine surfaces and arguments from arithmetic dynamics. We then show a dynamical Mordell-Lang type result for surfaces
in $X_0 \times X_0$.
\end{abstract}

\section{Introduction}\label{sec:intro}

\subsection{Loxodromic birational maps}\label{subsec:loxodromic-bir-maps}
Let $K$ be a field, $X$ be a projective surface over $K$. We write $\Bir (X)$ for the group of birational map of $X$, if $X$ is rational then $\Bir (X) = \Bir (\P^2_K)$ is the
Cremona group. Let $f$ is a birational map over $X$, the \emph{dynamical degree}
of $f$ is the number defined as
\begin{equation}
  \lambda (f) = \lim_n \left( (f^n)^* H \cdot H \right)^{1/n}
  \label{eq:<+label+>}
\end{equation}
where $H$ is an ample divisor on $X$. It is a well defined number and it does not depend on the choice of the ample
divisor $H$. Furthermore, if $\phi : X \dashrightarrow Y$ is a birational map, then
\begin{equation}
  \lambda (f) = \lambda (\phi \circ f \circ \phi^{-1}).
  \label{eq:<+label+>}
\end{equation}
Hence, the dynamical degree is a birational invariant.
We have $\lambda (f) \geq 1$ and we say that $f$ is \emph{loxodromic} if its
dynamical degree $\lambda (f)$ is $> 1$.

\subsection{Loxodromic automorphisms of affine surfaces}\label{subsec:loxodromic-auto}

Let $X_0$ be a normal affine surface over a field $K$. A \emph{completion} of $X_0$ is a normal projective surface $X$
with an open embedding $X_0 \hookrightarrow X$ such that $X \setminus X_0$ admits a smooth open neighbourhood. For any
completion $X$, we have a natural inclusion $\Aut (X_0) \subset \Bir (X)$. The
dynamical degree of $f$ is defined as the dynamical degree of the birational map $f : X \dashrightarrow X$. It does not
depend on the completion $X$ by the birational invariance. An automorphism is loxodromic if the induced birational map
is.

Over the affine plane $\A^2_K$, a polynomial automorphism is loxodromic if and
only if it is conjugated to a Henon automorphism. An affine surface that admits a loxodromic automorphism is always
rational (see Proposition \ref{prop:charc-alg-torus}) therefore we have a natural inclusion $\Aut (X_0) \subset \Bir
(\P^2)$ whenever $\Aut (X_0)$ admits a loxodromic element.

We give an important example of a subgroup of $\Aut(\A^2_K)$ when $K$ is of positive characteristic $p$. Let $A$ be the $K$
algebra of polynomials in the Frobenius map $z \mapsto z^p$. It is a non commutative algebra and we write $\GL_2
(A)$ for the set of $2 \times 2$ matrices with coefficients in $A$ that are invertible over $A$. It acts on $\A^2_K$ via
the following action
\begin{equation}
  \begin{pmatrix} a & b \\ c & d\end{pmatrix} \cdot (x,y) = \left( a(x) + b (y), c(x) + d(y) \right)
  \label{eq:<+label+>}
\end{equation}
and can be seen as a subgroup of $\Aut (\A^2_K) \subset \Bir (\P^2_K)$.
The group $\GL_2 (A)$ is the normaliser of the additive group $\G_a (K) \times \G_a (K)$ (acting on the plane by
translations) inside $\Bir (\P^2_K)$ the group of birational transformation of the projective plane. We write $\Aut_F
(\A^2_K) = \GL_2 (A) \ltimes (\G_a (K) \times \G_a (K))$, the letter $F$ stands for Frobenius.

\subsection{Loxodromic automorphisms with orbits meeting infinitely many times}\label{subsec:loxo-auto-with-common-orbit}
If $X$ is a quasiprojective variety over a field $K$, $f : X \rightarrow X$ is a dominant endomorphism and $p \in
X(K)$, we write $\OO_f (p)$ for the orbit of $p$ under the action of $f$, that is
\begin{equation}
  \OO_f (p) := \left\{ f^n (p) : n \in \A \right\}
  \label{eq:<+label+>}
\end{equation}
where $\A = \Z$ if $f$ is an automorphism and $\A = \Z_{\geq 0}$ otherwise.

In \cite{ghiocaLinearRelationsPolynomial2012}, Ghioca, Tucker and Zieve showed that if $f,g : \C \rightarrow \C$ are two
nonlinear polynomial maps such that there exists $p,q \in \C$ with $\OO_f (p) \cap \OO_g (q)$ infinite, then $f$ and $g$
must share a common iterate. It is natural to study this dynamical problem in higher dimension. In dimension 2, the
dynamics of birational maps has been extensively studied (see \cite{gizatullinRationalGsurfaces2007},
\cite{dillerDynamicsBimeromorphicMaps2001}, \cite{cantatDynamiqueAutomorphismesSurfaces2001}, ...), so it is natural to
ask what the analogue of the result of Ghioca, Tucker and Zieve could be. The analogue in dimension 2 of nonlinear
polynomial maps is loxodromic birational maps of $\P^2$. We focus on the subclass of loxodromic automorphisms of normal
affine surfaces. In \cite{abboudDynamicsEndomorphismsAffine2023}, the author has
classified the dynamics of loxodromic automorphisms of normal affine surface using valuative techniques. Using these
results we manage to answer the problem of orbits meeting infinitely many times.

\begin{thm}\label{thm:relation-orbits}
  Let $X_0$ be a normal affine surface over a field $K$ of characteristic zero and $f,g$ be two loxodromic
  automorphisms. If there exists $p,q \in X_0 (K)$ such that $\OO_f (p) \cap \OO_g (q)$ is infinite, then there exists
  $N,M \in \Z \setminus \left\{ 0 \right\}$ such that
\begin{equation}
  f^N = g^M.
  \label{eq:<+label+>}
\end{equation}
\end{thm}

We actually prove this theorem in any characteristic but the statement is a bit more technical (see Theorem
\ref{thm:relation-orbits-any-char}). Indeed, if $K$ is of positive characteristic, there is an extra case we have to
deal with. The normalizer in $\Aut (\A^2_K)$ of
the additive group $\G_a(K) \times \G_a (K)$ acting by translation on the affine plane is a subgroup for which we manage
to show Theorem \ref{thm:relation-orbits} only with an extra condition on the density of the set $\OO_f (p) \cap \OO_g
(q)$. See \S \ref{sec:aut-F-A2}.

\subsection{Strategy of proof}
It suffices to prove the theorem when $K$ is finitely generated over its prime field. Indeed, $K$
contains such a subfield $K_0$ over which $f,g,X_0,p,q$ are defined. Furthermore, if $K$ is a finite field then the
theorem is void so if $\car K > 0$ we will assume that $K$ has transcendence degree $\geq 1$ over its prime field.
Let us note the following characterisation of the algebraic torus, which was
proven in \cite{abboudDynamicsEndomorphismsAffine2023} \S 10.
\begin{prop}\label{prop:charc-alg-torus}
  If $X_0$ is a normal affine surface with a loxodromic automorphism, then $X_0$ is rational and we have the following
  dichotomy
  \begin{enumerate}
    \item $X_0 \simeq \G_m^2$.
    \item $K [X_0]^\times = K^\times$.
  \end{enumerate}
\end{prop}
Therefore a normal affine surface $X_0$ admitting a loxodromic automorphism must be rational, and we have a natural
embedding $\Aut (X_0) \hookrightarrow \Bir (\P^2)$ the group of birational transformations of $\P^2$.

In \cite{abboudDynamicsEndomorphismsAffine2023}, the author showed
that, when $X_0 \neq \G_m^2$ and $f \in \Aut (X_0)$ is a loxodromic automorphisms there are exactly two valuations $v_+,
v_-$ on the ring of regular functions of $X_0$ that are fixed by $f$. Using these two valuations, we can construct good
completions $X$ such that $f$ and $f^{-1}$ admits a locally attracting fixed point $p_+(X), p_-(X)$ at infinity that are
related to $v_+$ and $v_-$. From \cite{cantatGroupesTransformationsBirationnelles2011}, we have that a loxodromic
automorphism of $X_0$ induces a hyperbolic isometry of some infinite dimensional hyperbolic space with two fixed point
$\theta^+, \theta^-$ on the boundary. These two fixed points correspond exactly to the valuations $v_+$ and $v_-$.

With this preliminary work, the proof works as follows: If $\OO_f (p) \cap \OO_g (q)$ is infinite, then $p$ is not a
periodic point of $f$ and $q$ is not a periodic point of $g$. By an argument from arithmetic dynamics using Weil height
and the Northcott property we show that there
must exist an absolute value over $K$ such that both $\OO_f (p)$ and $\OO_g (q)$ are unbounded and we must have $f^n (p)
\rightarrow p_+^f (X)$ and $g^m (q) \rightarrow p^+_g (X)$. This implies that $p_+^f (X) = p_+^g (X)$ for any good
completion $X$. This means that $v_+ (f) = v_+ (g)$ and therefore $\theta^+ (f) = \theta^+ (g)$ and if $\langle f,g
\rangle$ is not conjugated to a subgroup of $\Aut_F (\A^2_K)$ this can only occur when $f^N = g^M$ by the work of Urech in
\cite{urechSubgroupsEllipticElements2021}.

For $\Aut (\G_m^2)$, the start of the proof is the same. We show that there must exists an absolute
value such that the orbits are unbounded. We show that $f^n (p)$ and $g^n (p)$ converge to the same point at infinity on
some completion $X$ at a speed $\simeq \lambda(f)^n$ (resp. $\simeq \lambda (g^n)$). This will imply that we must have $\lambda
(f)^a = \lambda (g)^b$ for some positive integers $a,b$ and the existence of an integer $c$ such that $f^{an}(p) =
g^{bn + c} (q)$ for infinitely many $n$. The conclusion will follow by applying known results on the dynamical
Mordell-Lang conjecture (see the next paragraph).

\subsection{A dynamical Mordell-Lang style result}\label{subsec:second-result}
The dynamical Mordell-Lang conjecture states the following
\begin{conj}\label{conj:mordell-lang}
  let $X$ be a quasiprojective variety over a field $K$ of characteristic zero and let $f : X \rightarrow X$ be an
  endomorphism of $X$. Let $V \subset X$ be a closed subvariety and $p \in X(K)$, then the set
  \begin{equation}
    \left\{ n \in \Z_{\geq 0} : f^n (x) \in V \right\}
    \label{eq:<+label+>}
  \end{equation}
  is the union of a finite set and a finite number of arithmetic progressions, i.e sets of the form $\left\{ a n + b : n
  \geq 0 \right\}$ with $a,b \in \Z_{\geq 0}$. In particular, the closure of $\OO_{f^a} (f^b (x))$ is $f^a$-invariant.
\end{conj}
It has been proven in \cite{bellDynamicalMordellLangProblem2010} for étale endomorphisms of quasiprojective varieties.
We prove two other results which are in the vein of the Dynamical Mordell-Lang conjecture. They are analogues of Theorem
1.4 and 1.5 of \cite{ghiocaLinearRelationsPolynomial2012}.
\begin{thm}\label{thm:dyn-mordell-lang-like}
  Let $X_0$ be a normal affine surface over a field $K$ of characteristic zero and $V \subset X_0 \times X_0$ a closed
  irreducible subvariety of dimension 2. Suppose there exists $(x_0, y_0) \in X_0 (K) \times X_0 (K)$ such that
  \begin{equation}
    \OO_{(f,g)} (x_0, y_0) \cap V
    \label{eq:<+label+>}
  \end{equation}
  is infinite, then $V$ is $(f,g)$-periodic.
\end{thm}
The proof of this results uses the dynamical Mordell-Lang conjecture for étale endomorphisms from
\cite{bellDynamicalMordellLangProblem2010} and the fact that a loxodromic automorphisms of an affine surface does not
admit invariant curves. Using Theorem \ref{thm:relation-orbits}, we also prove a slightly stronger result with more
assumption on $V$.

\begin{thm}\label{thm:dyn-mordell-lang-like-graph}
  Let $X_0$ be a normal affine surface over a field $K$ of characteristic zero and let $f,g,h \in \Aut (X_0)$ such that
  $f,g$ are loxodromic. Let $\Gamma_h \subset X_0 \times X_0$ be the graph of $h$ and suppose there exists $(x_0, y_0)
  \in X_0 (K) \times X_0 (K)$ such that
  \begin{equation}
    (\OO_f (x_0) \times \OO_g (y_0)) \cap \Gamma_h
    \label{eq:<+label+>}
  \end{equation}
  is infinite, then $\Gamma_h$ is $(f,g)-periodic$.
\end{thm}
In Theorem 1.5 of \cite{ghiocaLinearRelationsPolynomial2012}, the authors were considering lines inside $\C \times \C$
which are exactly graphs of affine automorphisms of $\C$ except for horizontal and vertical lines but in that case the
result is trivial.

\subsection*{Acknowledgements}\label{subsec:acknowledgments}
I thank John Lesieutre who told me about the problem of birational maps of $\P^2$ with orbits meeting infinitely
many times and Serge Cantat for some discussions related to this problem.

\section{Some preparations}\label{sec:preparations}
\subsection{Technical lemmas}\label{subsec:technical-lemmas}
We state here some technical lemmas for the proof of Theorem \ref{thm:relation-orbits}. For $f \in \Aut (X_0)$ and
$p \in X_0 (K)$, we define
\begin{equation}
  \OO_{f, +} (p) := \left\{ f^n (p) : n \in \Z_{\geq 0} \right\}, \quad \OO_{f, -} (p) := \left\{ f^n (p) : n \in
  \Z_{\leq 0} \right\}.
  \label{eq:<+label+>}
\end{equation}

If $S \subset \Z$ is a subset of integers, the \emph{Banach density} of $S$ is defined as
  \begin{equation}
    \delta (S) = \limsup_{\left| I \right| \rightarrow +\infty} \frac{\left| S \cap I \right|}{\left| I \right|}.
    \label{eq:<+label+>}
  \end{equation}
  where $I$ runs through intervals in $\Z$. In particular, if $S_1, \dots, S_r \subset \Z$, then
  \begin{equation}
    \delta (S_1 \cup \dots \cup S_r) \leq \delta (S_1) + \dots \delta (S_r).
    \label{eq:<+label+>}
  \end{equation}
  So, if $\delta (S_1 \cup \dots \cup S_r) > 0$, one of the $S_i$ must also be of positive Banach density.

  If $\OO_f (p) \cap \OO_g (q)$ is infinite, then we define a map $\iota: \OO_f (p) \cap \OO_g (q) \rightarrow \Z$ as
  follows. If $x = f^n (p) = g^m (q)$, then $\iota (x) = n$. This is a well defined injective map because $p$ cannot be
  $f$-periodic. So we can see $\OO_f (p) \cap \OO_g (q)$ as a subset of $\Z$ and define its Banach density using the map
  $\iota$.

\begin{lemme}\label{lemme:reductions}
  Let $X_0$ be a normal affine surface and $f,g \in \Aut (X_0)$ be two automorphisms such that there exists $p,q \in X_0
  (K)$ such that
  \begin{equation}
    \OO_f (p) \cap \OO_g (q)
    \label{eq:<+label+>}
  \end{equation}
  is infinite (resp. of positive Banach density). We have the following
  \begin{enumerate}
    \item \label{item:plus} Up to changing $f$ or $g$ by their inverse, we can suppose that $\OO_{f,+} (p) \cap
      \OO_{g,+} (p)$ is infinite (resp. of positive Banach density).
    \item \label{item:same-point} We can suppose that $p=q$.
    \item \label{item:iterate} For any $n,m \in \Z \setminus \left\{ 0 \right\}$, there exists $p ', q' \in X_0 (K)$ such that
      $\OO_f^n (p ') \cap \OO_g^m (q')$ is infinite (resp. of positive Banach density).
  \end{enumerate}
  Furthermore, if we assume that the Banach density of the set
  \begin{equation}
    \left\{ n \in \Z : \exists m \in \Z, f^n (p) = g^m (q) \right\}
\label{eq:<+label+>}
\end{equation}
is positive, then we can still assume positivity of the density after any of these 3 reductions.
\end{lemme}
\begin{proof}
  Item (1) follows from
  \begin{equation}
    \OO_f (p) \cap \OO_g (q) = \bigsqcup_{\epsilon_1, \epsilon_2 \in \left\{ +,- \right\}} \OO_{f, \epsilon_1}
    (p) \cap \OO_{g, \epsilon_2} (q).
    \label{eq:<+label+>}
  \end{equation}
  So, one of this four set is infinite (resp. of positive Banach density) and up to changing $f$ or $g$ by their inverse
  we can assume that $\OO_{f,+} (p) \cap \OO_{g, +} (q)$ is infinite (resp. of positive Banach density).

  For the proof of item (2), we first replace $f$ or $g$ by their inverse such that $\OO_{f,+} (p) \cap \OO_{g, +} (q)$
  is infinite using (1). Then, let $n_0, m_0 \in \Z_{\geq 0}$ be such that $f^{n_0} (p) = g^{m_0} (q)$. Define $r =
  f^{n_0} (p)$. Because $p$ is
  not $f$-periodic and $q$ is not $g$-periodic, the set
  \begin{equation}
    \left\{ (n,m) \in \Z_{\geq 0}^2 : f^n (p) = g^m (q), n < n_0 \text{ or } m < m_0 \right\}
    \label{eq:<+label+>}
  \end{equation}
  is finite, therefore the set
  \begin{equation}
    \left\{ (n,m) \in \Z_{\geq 0}^2 : f^n (p) = g^m (q), n \geq n_0, m \geq m_0 \right\} = \left\{ (n,m) \in
    Z_{\geq 0}^2 : f^{n - n_0} (r) = g^{m - m_0} (r) \right\}
    \label{eq:<+label+>}
  \end{equation}
  is infinite (resp. of positive Banach density). This shows (2).

  Item (3) follows from the equality
  \begin{equation}
    \OO_{f} (p) \cap \OO_g (q) = \bigsqcup_{l = 0, \dots, \left| n \right| -1} \bigsqcup_{k = 0, \dots, \left| m
    \right| -1} \OO_{f^n} (f^l (p)) \cap \OO_{g^m} (g^k (q))
    \label{eq:<+label+>}
  \end{equation}
  because one of these subsets must be infinite (resp. of positive Banach density).
\end{proof}

\subsection{Heights}\label{subsec:heights}
If $K$ is a number field, we define $\cM(K)$ as the set of normalised absolute values over $K$. They are either
archimedean and are defined as $\left| t \right| = \left| \sigma (t) \right|_\C$ where $\sigma : K \hookrightarrow \C$
is an embedding and $\left| \cdot \right|_\C$ is the usual absolute value over $\C$. Or they are non-archimedean if they
satisfy the inequality $\left| x+y \right| \leq \max (\left| x \right|, \left| y \right|)$. A non-archimedean absolute
value extends the $p$-adic absolute value over $\Q$ for some prime number $p$ and the normalisation is given by $\left| p
\right| = \frac{1}{p}$.

If $K$ is finitely generated over a finite field $\F$
such that $\tr.deg K / \F \geq 1$, we fix a normal projective variety $B$ over $\F$ with function field $K$. We
define $\cM (K)$ as the set of points of codimension 1 in $B$ (the set $\cM(K)$ depend on the choice of $B$ but this
won't have an importance so we omit it in the notation). Every such point is the generic point $\eta_E$ of an irreducible
codimension 1 subvariety $E \subset B$ and it induces an absolute value over $K$ as follows. The local ring at
$\eta_E$ is a discrete valuation ring and we write $\ord_E$ for the associated valuation which is the order of vanishing
at $\eta_E$. This defines the absolute value $\left| \cdot \right|_E = e^{-\ord_E (\cdot)}$ over $K$. Every absolute
value in $\cM(K)$ is non-archimedean in this case.

In either case, for an element $v \in \cM(K)$, we write $\left| \cdot \right|_v$ for the associated absolute value over
$K$ and we write $K_v$ for the completion of $K$ with respect to $v$. We call the \emph{Euclidian} topology, the
topology induced by $K_v$. In particular, if $X$ is a projective variety over
$K$, then $X(K_v)$ is a compact space. We define the associated norm over $K^n$
\begin{equation}
  \forall x = (x_1, \dots, x_n) \in K^n, \quad \left| \left| x \right| \right|_v := \max_i \left| x_i \right|_v.
  \label{eq:<+label+>}
\end{equation}
The \emph{Weil height} over $\A^n_K (K)$ is the function $h : \A_K^n (K) \rightarrow \R_{\geq 0}$ defined as
\begin{equation}
  h (p) = \sum_{v \in \cM(K)} \log^+  \left| \left| p \right| \right|_v
  \label{eq:<+label+>}
\end{equation}
where $\log^+ = \max(\log, 0)$.
The height function is well defined as for $t \in K$ there are only finitely many $v \in \cM(K)$ such that $\left| t
\right|_v \neq 1$. It satisfies the \emph{Northcott property}: for any $B \geq 0$, the set
\begin{equation}
  \left\{ p \in \A^n_K (K) : h(p) \leq B \right\}   \label{eq:<+label+>}
\end{equation}
is finite. In particular, if $f$ is a polynomial map over $\A^n$ and the sequence $(h(f^n(p)))$ is bounded, then
$\OO_f (p)$ is finite.

\subsection{Moriwaki height}\label{subsec:moriwaki-height}
If $K$ is a finitely generated field over $\Q$ with $\tr.deg K / \Q \geq 1$ we use Moriwaki heights. Here is briefly
how they work (see \cite{moriwakiArithmeticHeightFunctions2000} and
\cite{chenArithmeticIntersectionTheory2021}). Let $B$ be a normal and flat projective variety over $\spec \Z$ with
function field $K$. We define the set $\cM(K)$ of normalized absolute values over $K$ as follows.

For every prime number $p$, the fibre $B_p$ is the union of a finite number of irreducible components $\Gamma$ and except for a
finite number of prime $p$, $B_p$ is irreducible. Every such irreducible component is of codimension 1 in $B$, so we can define
the function $\ord_\Gamma : K^\times \rightarrow \Z$ which is the order of vanishing along $\Gamma$. Every $\Gamma$
obtained like this induces a non-archimedean absolute value over $K$ of the form
\begin{equation}
  \left| \lambda \right|_\Gamma = e^{- \ord_\Gamma}.
  \label{eq:<+label+>}
\end{equation}
If $\Gamma \subset B_p$, then $\left| \cdot \right|_\Gamma$ extends the $p$-adic absolute value over $\Q$.

For every $\C$-point $b \in B(\C)$, we have the archimedean absolute value
\begin{equation}
  \forall \lambda \in K, \quad \left| \lambda \right|_b := \left| \lambda (b) \right|_\C
  \label{eq:<+label+>}
\end{equation}
it is well defined if $b$ is not a pole of $\lambda$. We define
\begin{equation}
  \cM(K) = B (\C) \sqcup \bigsqcup_{p} \left\{ \left| \cdot \right|_\Gamma : \Gamma \subset B_p \right\}
  \label{eq:<+label+>}
\end{equation}
again $\cM (K)$ depends on $B$ but we omit it in the notation.

Write $d+1$ for the dimension of $B$ over $\spec \Z$. In particular, $\dim_\C B_\C = d$. An \emph{arithmetic polarisation} of
$K$ over $B$ in the sense of \cite{moriwakiArithmeticHeightFunctions2000} is the data of a big and nef
nef line bundle $\sL$ over $B$ and a plurisubharmonic metrisation of $L$ where $L$ is the line bundle induced by
$\sL$ over the analytic manifold $B(\C)$. We write $\overline L$ for the data of $L$ and its metrisation. It yields a finite
positive Borel measure $\mu_\C = c_1(\overline L)^{\dim B_\C}$ over $B(\C)$ of total mass $L^d$ and nonnegative numbers
$a_\Gamma := \sL_{|\Gamma}^{d}$ such that
\begin{enumerate}
  \item The measure $\mu_\C$ does not charge any algebraic subset.
\item For every prime number $p$,
\begin{equation}
  \sum_{\Gamma \subset B_p} a_\Gamma = L^d.
\label{eq:<+label+>}
\end{equation}
\end{enumerate}

\begin{ex}\label{ex:weil-metric}
  Let $B = \P^n_\Z$ with homogeneous coordinates $T_0, \dots, T_n$ and $\sL = \OO_B (1)$. We equip $\OO_{\P^d_\C} (1)$
  with the \emph{Weil} metric given by
  \begin{equation}
    \left| \left| a_0 T_0 + \dots + a_d T_d \right| \right| = \frac{\left| a_0 T_0 + \dots + a_d T_d \right|}{\max\left(
    \left| T_0 \right|, \dots, \left| T_d \right| \right)}.
    \label{eq:<+label+>}
  \end{equation}
  From \cite{chambert-loirHeightsMeasuresAnalytic2011} p.9, we have that the measure $\mu_\C = c_1
  (\overline{\OO(1)})^d$ is the Haar measure on the $n$-dimensional torus
  \begin{equation}
    (\CS^1)^n =
    \left\{ \left| T_0 \right| = \dots = \left| T_n \right| \right\}.
  \end{equation}
\end{ex}

\begin{lemme}\label{lemme:compact-support-measure-arithmetic-polarisation}
  Let $B_\Q$ be a normal projective variety over $\Q$ with function field $K$ and let $H$ be a very ample effective divisor over
  $B_\Q$, then there exists a flat normal projective variety $B$ over $\spec \Z$ with generic fibre $B_\Q$ and an arithmetic
  polarisation of $K$ over $B$ such that $\mu_\C$ has compact support in $B(\C) \setminus H(\C)$ \end{lemme}
\begin{proof}
  Let $B_\Q \hookrightarrow \P^N_\Q$ be an embedding such that $H$ is the intersection of of $B_\Q$ with the
  hyperplane $T_0 = 0$. We can find a normal flat projective variety $B$ over $\spec \Z$ with an embedding $B
  \hookrightarrow \P^N_\Z$ such that the generic fibre is $B_\Q \hookrightarrow \P^N_\Q$. The pull-back of
  $\OO_{\P^N_\Z}(1)$ equipped with the Weil metric induces an arithmetic polarisation of $K$ over $B$ and the
  support of $\mu_C$ is a compact subset of $B (\C) \setminus H(\C)$ contained in $\left\{ \left| t_1 \right| = \dots =
  \left| t_N \right| = 1 \right\} \subset \A^N (\C)$ where $t_i = \frac{T_i}{T_0}$ .
\end{proof}

The Weil height over $\A^N_K (K)$ associated to this arithmetic polarisation of $K$ over $B$ is
\begin{equation}
  \forall x \in \A^N (K), \quad h (x) = \sum_\Gamma a_\Gamma \log^+ \left| \left| x \right| \right|_\Gamma +
  \int_{B(\C)} \log^+ \left| \left| x(b) \right| \right|_\C d \mu_\C (b).
  \label{eq:<+label+>}
\end{equation}
The integral is well defined because $\mu_\C$ does not charge algebraic subsets, therefore the union of all the poles
over $B(\C)$ of every $\lambda \in K$ has $\mu_\C$ measure zero because $K$ is countable.
It also satisfies the Northcott property: the set
\begin{equation}
  \left\{x \in \P^N (K) : h (x) \leq A  \right\}
  \label{eq:<+label+>}
\end{equation}
is finite.

\subsection{Families of varieties}\label{subsec:family-of-varieties}
Let $K$ be a finitely generated field over $\Q$ with $\tr.deg K/\Q \geq 1$ and $X$ a quasiprojective variety over $K$. A
\emph{model} of $X$ is a projective morphism $q: \sX \rightarrow B$ between quasiprojective varieties over $\Q$
such that the generic fibre is isomorphic to $X \rightarrow \spec K$, in particular $K$ is the function field of $B$.
There are two types of irreducible subvarieties $\sY$ in $\sX$:
\begin{itemize}
  \item Horizontal subvarieties, they satisfy $q(\sY) = B$ and if $Y = \sY_{|X}$ is the generic fibre, then $\sY$ is
    exactly the closure of $Y$ in $\sX$.
  \item Vertical subvarieties, they satisfy $q(\sY) \neq B$.
\end{itemize}
If $p \in X(K)$, then $p$ induces a rational map $p : B \dashrightarrow \sX$. If $V \subset B$ is an open subset over
which $p$ is defined, then we write $p(V)$ for the image of $p$ in $\sX_V = \sX \times_B V = q^{-1}(V)$. The subvariety $p(V)$ is
also the closure of $p \in X(K)$ in $\sX_V$.
We use similar notations for the analytic manifolds $\sX (\C) \rightarrow B(\C)$.
Finally, if $f :X \rightarrow X$ is a dominant endomorphism, then there exists an open subset $V \subset B$ such that $f$ extends
to a dominant endomorphism $f : \sX_V \rightarrow \sX_V$.

\section{Valuations and Picard-Manin space}\label{sec:valuations-picard-manin}
\subsection{Picard-Manin space of a projective surface}\label{subsec:picard-manin}
Let $X$ be a rational projective surface. The Picard-Manin space of $X$ is defined as follows, consider
\begin{equation}
  \cC (X) = \varinjlim_{Y \rightarrow X} \NS (Y)_\R
  \label{eq:<+label+>}
\end{equation}
where the direct limit is over every projective surface with a birational morphism $\pi: Y \rightarrow X$, the compatibility
morphisms are given by the pull back morphisms $\pi^* : \NS (X)_\R \hookrightarrow \NS(Y)_\R$. We also define
\begin{equation}
  \cW (X) = \varprojlim_{Y \rightarrow X} \NS (Y)_\R
  \label{eq:<+label+>}
\end{equation}
where here the compatibility morphisms are given by the pushforward morphisms $\pi_* : \NS (Y) \rightarrow \NS
(X)$.

The intersection form on every $\NS(Y)_\R$ is of Minkowski type and induces a non degenerate bilinear form over
$\cC (X)$ with signature $(1, \infty)$. The \emph{Picard-Manin} space $\overline \cC(X)$ of $X$ is defined as the
Hilbert completion with respect to this intersection form. The hyperbolic space $\HH_X^\infty$ is defined as the
hyperboloid
\begin{equation}
  \HH_X^\infty = \left\{ Z \in \overline \cC (X) : Z^2 = 1, \quad Z \cdot H > 0 \right\}
  \label{eq:<+label+>}
\end{equation}
where $H$ is a fixed ample class in some model over $X$. The boundary of this hyperbolic space is $\mathbb P \left\{
Z : Z^2 = 0, Z \cdot H > 0 \right\}$.

A birational map $f \in \Bir(X)$ acts by pull-back on $\cC(X)$ and also over $\HH_X^\infty$. The action over
$\HH_X^\infty$ is by isometries. We have the following theorem
\begin{thm}[\cite{cantatGroupesTransformationsBirationnelles2011}]\label{thm:cantat-hyperbolic}
  We have the following classification
  \begin{enumerate}
    \item If $\lambda (f) > 1$, $f$ is \emph{loxodromic} it admits exactly two distinct fixed point $\theta^+, \theta^-
      \in \partial \HH_X^\infty$ such that $f^* \theta^+ = \lambda (f) \theta^+$ and $f^* \theta^- = 1 /
      \lambda(f) \theta^-$.
    \item If $\lambda (f) = 1$ and $f^*$ has no fixed point in $\HH_X^\infty$, then $f$ is \emph{parabolic}, then it
      admits a unique fixed point $\theta \in \partial
      \HH_X^\infty$ such that $f^* \theta = \theta$ and $\theta \in \cC$ is the class of an invariant fibration.
    \item If $\lambda (f) = 1$ and $f^*$ has a fixed point in $\HH_X^\infty$ then $f$ is \emph{elliptic} and there
      exists a model $Y \rightarrow X$ such that $f_Y$ is an automorphism.
  \end{enumerate}
\end{thm}

If $\theta \in \partial \HH_X^\infty$, we write $\Stab (\theta)$ for the subgroup of $\Bir (\P^2)$ defined by
\begin{equation}
  \Stab (\theta) = \left\{ f \in \Bir (\P^2) : \exists t_f > 0, f^* \theta = t_f \theta \right\}.
  \label{eq:<+label+>}
\end{equation}

We give an important example of a subgroup of $\Aut(\A^2_K) \subset \Bir (\P^2_K)$ when $K$ is of positive
characteristic $p$. Let $A$ be the $K$ algebra of polynomials in the Frobenius map $z \mapsto z^p$. It is a
noncommutative algebra and we write $\GL_2 (A)$ for the set of $2 \times 2$ matrices with coefficients in $A$ that are
invertible over $A$. It acts on $\A^2_K$ via the following action
\begin{equation}
  \begin{pmatrix} a & b \\ c & d\end{pmatrix} \cdot (x,y) = \left( a(x) + b (y), c(x) + d(y) \right)
  \label{eq:<+label+>}
\end{equation}
and can be seen as a subgroup of $\Aut (\A^2_K) \subset \Bir (\P^2_K)$.
The group $\GL_2 (A)$ is the normaliser of the additive group $\G_a (K) \times \G_a (K)$ (acting on the plane by
translations) inside $\Bir (\P^2_K)$ the group of birational transformation of the projective plane. We define
\begin{equation}
  \Aut_F (\A^2_K) := \GL_2 (A) \ltimes (\G_a (K) \times \G_a (K)).
\end{equation}
The letter $F$ stands for Frobenius. In particular, for any
loxodromic automorphism $f \in \Aut_F (\A^2_K)$, the group $\G_a (K) \times \G_a (K)$ fixes $\theta_f^+$ and
$\theta_f^-$.

\begin{prop}[\cite{urechSubgroupsEllipticElements2021}]\label{prop:stab-theta}
  Let $f \in \Bir (\P^2)$ be a loxodromic birational transformation, then $\Stab (\theta^+_f)$ contains $\langle f
  \rangle$ as a subgroup of finite index unless $f$ is conjugated to
  \begin{enumerate}
    \item an automorphism of the algebraic torus.
    \item \label{item:automorphism-affine} an automorphism of $\A^2_K$ of the form $g (x,y) = \left( a(x) + b (y), c(x)
      + d (y) \right)$ where $a,b,c,d \in K[t^p]$ and $\car K = p > 0.$
  \end{enumerate}
\end{prop}
\begin{proof}
  This is exactly the content of Lemma 7.3 of \cite{urechSubgroupsEllipticElements2021} in the characteristic zero case
  which is based on Theorem 7.1 of \emph{loc.cit} that states the following:

  {If $0 \rightarrow H \rightarrow N
  \rightarrow A \rightarrow 0$ is an exact sequence of $\Bir (\P^2_\C)$ such that $N$ contains a loxodromic element and
$H$ is an infinite subgroup of elliptic elements, then $N$ in conjugated to a subgroup of $\Aut (\G_m^2)$}.

In positive characteristic, this result also holds but $H$ can also be conjugated to $\Aut_F (\A^2_K)$. The rest of the
proof is unchanged, see \cite{cantatCremonaGroup2018} Example 7.2, Theorem 3.3 and Remark 7.4.
\end{proof}

\begin{rmq}\label{rmq:dynamical-degrees-discrete-subgroup}
  For any $g \in \Stab (\theta^+_f)$ we must have that $t_g$ or $t_g^{-1}$ is the dynamical degree of $g$. We have a
  group homomorphism $g \in \Stab (\theta^+_f) \mapsto \log t_g \in \R$ and its image must be a discrete subgroup of
  $\R$ because of the spectral gap property of the dynamical degrees of elements of $\Bir (\P^2)$ (see
  \cite{blancDynamicalDegreesBirational2013}). In particular, there exists $a,b \in \Z \setminus \left\{ 0
  \right\}$ such that $\lambda (f)^a = \lambda (g)^b$.
\end{rmq}

\subsection{For an affine surface}\label{subsec:for-affine-surface}
Now, Let $X_0$ is a normal rational affine surface over a field $K$ and let $X$ be a completion of $X_0$. The complement
$X \setminus X_0$ is a finite union of irreducible curves. We write $\DivInf (X)_\A$ for the set of $\A$-divisors
supported at infinity in $X$ with $\A = \Z, \Q,\R$. If $X_0$ is rational and $K [X_0]^\times = K^\times$, then we have
the injective group homomorphism
\begin{equation}
  \DivInf(X)_\A \hookrightarrow \NS (X)_\A.
  \label{eq:injective-group-homorphism}
\end{equation}
Indeed, it suffices to prove it for $\A = \Z$. We have $\DivInf (X) \hookrightarrow \Pic(X)$ because there is no
nonconstant invertible function over $X_0$ and then $\Pic (X) = \NS (X)$ because $X$ is rational. We define the
following spaces
\begin{equation}
  \cC (X_0) = \varinjlim_{Y} \NS (Y)_\R, \quad \cW (X_0) = \varprojlim_Y \NS (Y)_\R
  \label{eq:<+label+>}
\end{equation}
where the limits are over every completion $Y$ of $X_0$. If $X$ is a fixed completion of $X_0$ we have a canonical
surjective group homomorphisms $\cC (X) \twoheadrightarrow \cC(X_0)$ and $\cW(X) \twoheadrightarrow \cW (X_0)$.
We define the Picard-Manin space $\overline \cC (X_0)$ and the
hyperbolic space $\HH_{X_0}^\infty$ in the same fashion as for the projective surfaces. The group $\Aut (X_0)$ acts by
isometries over $\HH_{X_0}^\infty$ and Theorem \ref{thm:cantat-hyperbolic} also holds in this setting.

Finally, we introduce
\begin{equation}
  \Cinf = \varinjlim_Y \DivInf(Y)_\R, \quad \Winf = \varprojlim_Y \DivInf (Y)_\R
  \label{eq:<+label+>}
\end{equation}
and we have the following commutative diagram
\begin{equation}
  \begin{tikzcd}
    \Cinf \ar[r, hook] \ar[d, hook] & \cC (X_0) \ar[d, hook] \\
    \Winf \ar[r, hook] & \cW (X_0)
  \end{tikzcd}
  \label{eq:<+label+>}
\end{equation}
where the horizontal arrows are injective by \eqref{eq:injective-group-homorphism}.

\subsection{Valuations and divisors}\label{subsec:valuation-divisors}
Let $X_0$ be a normal affine surface over a field $K$ and denote by $A$ its ring of regular functions. A
\emph{valuation} over $A$ is a function $v : A \rightarrow \R \cup \left\{ +\infty \right\}$ such that
\begin{enumerate}
  \item $\forall P,Q \in K[X_0], \quad v(PQ) = v(P) + v(Q), \quad v(P+Q) \geq \min (v(P), v(Q))$.
  \item $v(0) = + \infty$.
  \item $v_{|K^\times} = 0$.
\end{enumerate}
Any automorphism $f \in \Aut (X_0)$ acts by pushforward
\begin{equation}
  f_* v (P) = v (f^* P).
  \label{eq:<+label+>}
\end{equation}
If $X$ is a completion of $X_0$, by the valuative criterion of properness there exists a unique (scheme) point $p \in X$
such that $v_{|\OO_{X,p}} \geq 0$ and $v_{|\m_{X,p}} > 0$, we call this point the \emph{center} of $v$ in $X$ and write
it $c_X (v)$. If $\pi : Y \rightarrow X$ is another completion above $X$, then
\begin{equation}
  c_Y (v) \in \pi^{-1} (c_X (p)), \quad \pi (c_Y (v)) = c_X (v).
  \label{eq:<+label+>}
\end{equation}
Two valuations $v,w$ are proportional if and only if for any completion $X$ we have $c_X (v) = c_X (w)$ and it suffices
to check the equality for a cofinal set of completion $X$.

We write $\cV$ for the space of valuations over $A$ and $\cV_\infty$ for the set of valuations centered at infinity.
\begin{ex}\label{ex:valuations}
  We give two examples of valuations. If $X$ is a completion of $X_0$ and $E$ is an irreducible curve in $X$, then the
  local ring at the generic point $\eta_E$ of $E$ is a discrete valuation ring with valuation $\ord_E$, the order of vanishing
  along $E$. This induces a valuation over $K[X_0]$ and $c_X (\ord_E) = \eta_E$. The valuation $\ord_E$ is centered at
  infinity if and only if $E$ is one of the irreducible component of $X \setminus X_0$. Any valuation proportional to
  some $\ord_E$ is called \emph{divisorial}.

  Let $p \in X (K)$ be a closed point and assume that $X$ is regular at $p$. Let $(x,y)$ be local coordinates at $p$,
  then for any $\alpha, \beta >0$ we define the valuation $v_{\alpha, \beta}$ at the local ring of $p$ by
  \begin{equation}
    v_{\alpha,\beta} \left( \sum_{i,j}a_{ij} x^i y^j \right) = \min \left( \alpha i + \beta j : a_{ij} \neq 0 \right).
    \label{eq:<+label+>}
  \end{equation}
This induces a valuation over $K\left[ X_0 \right]$ because any regular function can be expressed as a quotient of
germs of regular functions at $p$. We have $c_X (v_{\alpha, \beta}) = p$ and it is centered at infinity if and only if
$p \not \in X_0 (K)$. These valuations are called \emph{monomial} valuations.
\end{ex}

In \cite{abboudDynamicsEndomorphismsAffine2023}, we showed that every $v \in \cV_\infty$ induces a Weil
divisor $Z_v \in \Winf$ with the property that
\begin{equation}
  f_* Z_v = Z_{f_* v}.
  \label{eq:<+label+>}
\end{equation}
And we have the following
\begin{thm}[\cite{abboudDynamicsEndomorphismsAffine2023}, Theorem 11.16 and Theorem 14.4]\label{thm:nord-sud}
  If $f$ is a loxodromic automorphism with two fixed points $\theta^+_f, \theta^-_f \in \partial H_{X_0}^\infty$, then
  $\theta^+_f, \theta^-_f \in \overline \cC (X_0) \cap \Winf$ and these two divisors correspond to two eigenvaluations $v_+,
  v_- \in \cV_\infty$ such that
  \begin{equation}
    Z_{v_{+}} = \theta^-_f, \quad Z_{v_-} = \theta^+_f
    \label{eq:<+label+>}
  \end{equation}
  and for any completion $X$ of $X_0$, $c_X (v_\pm)$ is a closed point.
\end{thm}

\begin{cor}\label{cor:same-eigenvaluations}
  If two loxodromic automorphisms of an affine surface $X_0$ have the same eigenvaluation $v_-$ then they must share a
  common iterate unless
  \begin{enumerate}
    \item $X_0 \simeq \G_m^2$.
    \item $\car K > 0, X_0 \simeq \A^2_K$ and $f,g \in \Aut_F(\A^2_K)$.
  \end{enumerate}
\end{cor}
\begin{proof}
  Since $Z_{v_-} = \theta^+$, this follows directly from Proposition \ref{prop:stab-theta}.
\end{proof}

\subsection{Dynamics of a loxodromic automorphism}\label{subsec:dyn-loxo-auto}
Finally we have this result on the dynamics of a loxodromic automorphism.
\begin{thm}[\cite{abboudDynamicsEndomorphismsAffine2023}, Theorem 14.4]\label{thm:dynamics-loxodromic-automorphisms}
  Let $X_0$ be a normal affine surface over a field $K$ and $f \in \Aut (X_0)$ a loxodromic automorphism, then there
  exists a completion $X$ of $X_0$ and closed points $p_+, p_- \in X (K)$ such that
  \begin{enumerate}
    \item $p_+ \neq p_-.$
    \item $f^{\pm 1}$ is defined at $p_\pm$ and $f^{\pm} (p_\pm) = p_\pm$.
    \item There exists $N_0$ such that $\forall N \geq N_0, f^{\pm N}$ contracts $X \setminus X_0$ to $p_\pm$.
    \item \label{item:local-normal-forms} There exists local coordinates $(u,v)$ at $p_+$ such that
      \begin{equation}
        f(u,v) = \left(\alpha(u,v) u^a v^b, \beta(u,v) u^c v^d\right) \text{ or } f(u,v) = \left(\phi (u,v) u^a, u^b v
          \psi_1 (u,v) + \psi_2 (u)\right)
        \label{eq:<+label+>}
      \end{equation}
      with $\alpha, \beta$ invertible and $a,b,c,d \geq 1$ or $\phi$ is invertible, $\psi_1 (0, v) \neq 0$ and $\psi_2
      (0) \neq 0$, $a \geq 2$ and $b \geq 1$.

      In the first case, $uv = 0$ is a local equation of $X \setminus X_0$ at
      $p_+$ and in the second case, $u = 0$ is a local equation of $X \setminus X_0$ at $p_+$.
      The analogue statement holds for $p_-$ and $f^{-1}$.
    \item In particular, If $K \hookrightarrow K_v$ is an embedding into a complete field, then there exists a basis of
      small open neighbourhood $U^\pm$ of $p_\pm$ in $X (K_v)$ for the Euclidian topology such that $f^{\pm} (U^\pm)
      \Subset U^\pm$ and for every $x \in U^\pm, f^{\pm k} (x) \xrightarrow[k \rightarrow +\infty]{} p_\pm$.
    \item $p_\pm = c_X (v_\mp)$.
  \end{enumerate}
  Furthermore, any completion obtained by blowing up $X$ at infinity satisfies the same properties.
\end{thm}
A completion that satisfies Theorem \ref{thm:dynamics-loxodromic-automorphisms} will be called a \emph{dynamical
completion} of $f$.

\begin{lemme}\label{lemme:convergence-vers-p}
  Let $f$ be a loxodromic automorphism of $X_0$ and $X$ be a dynamical completion of $f$. If there exists an absolute
  value $v \in \cM(K)$ and $x \in X_0(K_v)$ such that the forward $f$-orbit of $x$ is unbounded, then $f^n (x)
  \xrightarrow[n \rightarrow + \infty]{} p_+$.
\end{lemme}
\begin{proof}
  We use the notations of Theorem \ref{thm:dynamics-loxodromic-automorphisms} (5). Since $X (K_v)$ is compact the
  sequence $f^n (x)$ must accumulate to a point $q \in X \setminus X_0 (K_v)$. If $q \neq
  p_-$, then since $f^{l} (q) = p_+$ for $l$ large enough we must have that there exists $n_0$ such that $f^{n_0} (x) \in
  U^+$ and therefore $f^n (x) \rightarrow p_+$.

  Otherwise we must have $f^n (x) \rightarrow p_-$ but this is not possible since we would have that for every open
  small neighbourhood $U^-$ of $p_-$ in $X(K_v)$ there exists $n_0$ such that $f^{n_0} (x) \in U^-$, but since $U^-$ is
  $f^{-1}$-invariant we get $x \in U^-$ for arbitrary small open neighbourhood $U^-$ of $p_-$, this is absurd.
\end{proof}

\begin{cor}\label{cor:local}
  Let $K_v$ be a complete field and $X_0$ a normal affine surface over $K_v$. Suppose that there exists $p,q \in X_0
  (K_v)$ such that
  \begin{enumerate}
    \item $\OO_{f,+} (p) \cap \OO_{g,+} (q)$ is infinite.
    \item $\OO_{f,+} (p)$ is unbounded in $X_0 (K_v)$, meaning its closure is not compact.
    \item $X_0 \not \simeq \G_m^2$.
  \end{enumerate}
  Then, $f,g$ have the same eigenvaluation $v_-$. Furthermore, if $\langle f,g \rangle$ is not conjugated to a subgroup
  of $\Aut_F (\A^2_K)$ in $\Bir (\P^2_K)$, then there exists $N,M \neq 0$ such that
  \begin{equation}
    f^N = g^M.
    \label{eq:<+label+>}
  \end{equation}
\end{cor}
\begin{proof}
  For any completion $X$ of $X_0$ that satisfies Theorem
\ref{thm:dynamics-loxodromic-automorphisms}, we must have by Lemma \ref{lemme:convergence-vers-p} that $p_+ = c_X
(v_- (f)) = c_X (v_- (g))$. This means by Theorem \ref{thm:dynamics-loxodromic-automorphisms} that for a cofinal set of
completions $X$, we have $c_X (v_-) (f) = c_X (v_- (g))$. Thus $v_- (f) = v_- (g)$ and we conclude by Corollary
\ref{cor:same-eigenvaluations}.
\end{proof}

\section{Proof of Theorem \ref{thm:relation-orbits}}\label{sec:proof}
We can now state Theorem \ref{thm:relation-orbits} in any characteristic.
\begin{thm}\label{thm:relation-orbits-any-char}
  Let $X_0$ be a normal affine surface over a field $K$ of any characteristic and $f,g$ be two loxodromic automorphisms
  and suppose that $\langle f,g \rangle$ is not conjugated to a subgroup of $\Aut_F (\A^2_K)$ in $\Bir (\P^2_K)$. If
  there exists $p,q \in X_0 (K)$ such that $\OO_f (p) \cap \OO_g (q)$ is infinite, then there exists $N,M \in \Z
  \setminus \left\{ 0 \right\}$ such that
\begin{equation}
  f^N = g^M.
  \label{eq:<+label+>}
\end{equation}

If $\langle f,g \rangle$ is conjugated to a subgroup of $\Aut_F (\A^2_K)$ and the set
    \begin{equation}
      \left\{ n \in \Z : \exists m \in \Z,  f^n (p) = g^m (q) \right\}
      \label{eq:<+label+>}
    \end{equation}
    is of positive Banach density, then the same conclusion holds.
\end{thm}

We first prove the theorem when $X_0 \neq \G_m^2$ and $\langle f,g \rangle$ is not conjugated to a subgroup of
$\Aut_F (\A^2_K)$. We assume that $K$ is finitely generated over its prime field. By Lemma \ref{lemme:reductions}
\ref{item:plus}, we can suppose that $\OO_{f,+} (p) \cap \OO_{g, +} (q)$ is infinite. If $\langle f,g\rangle$ is not conjugated
to a subgroup of $\Aut_F (\A^2_K)$, Theorem
\ref{thm:relation-orbits-any-char} follows from Corollary \ref{cor:local} and the following lemma.
\begin{lemme}\label{lemme:place-orbit-unbounded}
  If $f$ is a loxodromic automorphism of $X_0 \not \simeq \G_m^2$ and $p \in X_0 (K)$ is not $f$-periodic, then there exists a
  place $v \in \cM (K)$ such that the forward orbit $\OO_{f,+} (p) \subset X_0 (K_v)$ is unbounded.
\end{lemme}
\begin{proof}
  Let $X$ be a dynamical completion of $f$. By a result of Goodman in \cite{goodmanAffineOpenSubsets1969}, there exists an ample
  effective divisor $H$ supported at infinity. We can suppose that $H$ is very ample to get an embedding $X
  \hookrightarrow \P^N_K$ such that $H$ is the restriction of the hyperplane $T_0 = 0$ where $T_0, \dots T_N$ are the
  homogeneous coordinates of $\P^N$. Then, we have an embedding of $X_0$ into $\A^N$ with affine coordinates $t_i :=
  \frac{T_i}{T_0}$. We have that there exists a polynomial endomorphism $u : \A^N_K \rightarrow \A^N_K$ such that $u$
  restricts to $f$ over $X_0$. Indeed the ring of regular functions of $X_0$ is of the form $\K [t_1, \dots, t_N] /
  (f_1, \dots, f_r)$ so we can lift any ring endomorphism of $X_0$ to an endomorphism of $\A^N_K$.

  Now the proof differs whether $K$ is transcendental over $\Q$ or not so we split the two cases in the next
  subsections.
\end{proof}

\subsection{Proof of the lemma when $K$ is a number field or $\car K > 0$.}\label{subsec:first-proof}
  Suppose $K$ is a number field or of positive characteristic.
  If $p \in X_0 (K)$, then for all but finitely many $v \in \cM (K)$ we have
  \begin{equation}
    \forall k \geq 0, \quad \log^+ \left| \left| u^k (p) \right| \right|_v = 0.
    \label{eq:<+label+>}
  \end{equation}
  Indeed, we can remove every non-archimedean $v$ such that all the coefficients of $u$ and all the coordinates of $p$
  have absolute
  value 1. Let $h$ be the Weil height over $\A^N_K (K)$. By the Northcott property, if the forward orbit $\OO_{f,+} (x)$
  is infinite, then the heights $h (f^k (x))$ must be unbounded and therefore there must exists $v$ such that the
  sequence $\log^+ \left| \left| f^k (p) \right| \right|_v $ is unbounded.

  \subsection{Proof of the lemma when $\tr.deg K /\Q \geq 1$}\label{subsec:second-proof}

The only remaining case in the proof is when $K$ is a field of characteristic 0 not algebraic over $\Q$. We have assumed
that $K$ is finitely generated over $\Q$ so that we can use use Moriwaki height. We will first do some preparations,
then choose an arithmetic polarisation of $K$ suitable to our needs.

Let $B_\Q$ be a normal projective variety over $\Q$ with function field $K$.
Consider the embedding $X \hookrightarrow \P^N_\Q
\hookrightarrow \P^N_{B_\Q}$. We write $\sX$ for the closure of $X$ in $\P^N_{B_\Q}$ and $q : \sX \rightarrow B_\Q$ for the
projective structure morphism. The automorphisms $f$ and $f^{-1}$ extend to birational maps $f^{\pm 1} : \sX
\dashrightarrow \sX$.  Let $\sV \subset \sX$ be the union of the vertical components of $\Ind (f : \sX
\dashrightarrow \sX)$ and $\Ind (f^{-1} : \sX \dashrightarrow \sX)$, then $q (\sV)$ is a strict closed subvariety of $B_\Q$.
Let $\Lambda \subset B$ be an open subset such that
\begin{enumerate}
  \item $\Lambda \cap q(\sV) = \emptyset$.
  \item $q_\Lambda : q^{-1}(\Lambda) \rightarrow \Lambda$ is flat.
  \item For every $\lambda \in \Lambda$, $q^{-1}(\lambda)$ is irreducible.
  \item The point $p \in X_0 (K) \subset X(K)$ defines a regular map $p : \Lambda \rightarrow \sX$.
\end{enumerate}
Let $(u,v)$ be local coordinates at $p_+$ in $X$ such that Theorem \ref{thm:dynamics-loxodromic-automorphisms}
\ref{item:local-normal-forms} holds.
Let $U,V$ be two rational functions over $\sX$ such that $U_{|X} = u$ and $V_{|X} = v$. There exists an open
affine neighbourhood $O^+$ of $p_+$ in $q^{-1}(\Lambda)$ such that $U,V$ are regular over $O^+$, $U = V = 0$ is the equation of
$p_+ (\Lambda) \cap O^+$ in $O^+$ and
Theorem \ref{thm:dynamics-loxodromic-automorphisms} \ref{item:local-normal-forms} holds with $U,V$ and regular
(resp. invertible regular) functions $\alpha,\beta,\phi, \psi_1, \psi_2$ over $\OO_+$. We define an affine neighbourhood $O^-$
of $p_-$ in $q^{-1}(\Lambda)$ similarly using the local normal form of $f^{-1}$ at $p_-$. Let $\sZ$ be the union of the
vertical components of $\sX \setminus O^+ \cup \sX \setminus O^-$, we replace $\Lambda$ by $\Lambda \setminus q (\sZ)$.

Now, let $H$ be a very ample effective divisor over $B_\Q$ such that the support of $H$ contains $B_\Q \setminus \Lambda$.
By Lemma \ref{lemme:compact-support-measure-arithmetic-polarisation}, there exists a normal projective variety $B$ over
$\spec \Z$ with generic fibre $B_\Q$ and an arithmetic polarisation of $K$ over
$B$ such that $\mu_\C$ is a compact subset of $B(\C) \setminus H (\C)$, hence a compact subset of $\Lambda (\C)$. Let
$h$ be the associated Weil height over $\A^N(K)$.

The sequence $(h(f^n) (p))_{n \geq 0}$ is unbounded because otherwise $p$ would be $f$-periodic by the Northcott property. The
height function is of the form
\begin{equation}
  h (x) = \sum_{\Gamma \subset B} a_\Gamma \log^+ \left| \left| x \right| \right|_\Gamma +
  \int_{B(\C)} \log^+ \left| \left| x(b) \right| \right|_\C d \mu_\C (b).
  \label{eq:<+label+>}
\end{equation}
By the same argument as in the previous paragraph for all but finitely many $\Gamma \subset B$ we have
\begin{equation}
  \log^+ \left| \left| f^n (p) \right| \right|_\Gamma = 0, \forall n \geq 0.
\end{equation}
Indeed, this will be true for any $\Gamma$ such that the coefficients of $u$ and the coordinates of $p$ have
$\Gamma$-absolute values equal to 1. Thus, we have two possibilities
\begin{enumerate}
  \item There exists $\Gamma$ such that $\log^+ \left| \left| f^n (p) \right| \right|_\Gamma$ is unbounded. In that
    case, $v = e^{-\ord_\Gamma}$ is the desired absolute value.
\item The sequence of integrals $\int_{B(\C)} \log^+ \left| \left| f^n (p(b)) \right|  \right| d \mu_\C (b)$
  is unbounded.
\end{enumerate}
For the second case, a priori we could have that for every $b \in B(\C)$, the sequence $\left| \left| f^n (p) \right|
\right|$ is bounded. We show this is not the case using the normal form of $f$ at $p_+$. Recall the definitions of $O^+$
and $O^-$. Let $S$ be the support of $\mu_\C$ over $\Lambda (\C)$, since $q : \sX
(\C) \rightarrow B(\C)$ is a proper map, $q^{-1} (S)$ is a compact subset of $\sX (\C)$. We denote it by $\sX (S)$.
The set $W_\epsilon^+ := \left\{x \in O^+ (\C) \cap \sX (S) : U (x), V (x) < \epsilon
\right\}$ is a relatively compact open neighbourhood of $p^+ (S)$ in $\sX(S)$. The functions $\alpha, \beta, \phi, \psi_1,
\psi_2$ appearing in the
local normal form of $f$ over $O^+$ are bounded over $W_\epsilon^+$ because it
is relatively compact in $O^+ (\C)$ and therefore for $\epsilon > 0$ small enough, $W_\epsilon^+$ is $f$-invariant and
if $x \in W_\epsilon^+ \cap q^{-1}(s)$ for some $s \in S$, then $f^n (x) \rightarrow
p_+(s)$ with respect to $\left| \left| \cdot \right| \right|_s$. Similarly, we can find a
relatively compact open neighbourhood of $p_- (S)$ in $\sX(S)$ which is $f^{-1}$-invariant.
 We can suppose up to shrinking $W_\epsilon^-$ that
\begin{equation}
p(S) \cap W_\epsilon^- = \emptyset
  \label{eq:condition-intersection}
\end{equation}
because $S$ is compact.

Furthermore, let $Y = X \setminus X_0$ and $\sY$ be the closure of $Y$ in $\sX$, we write $\sY (S) := \sY (\C) \cap
q^{-1}(S)$. The set
\begin{equation}
  f^{-1} (W_\epsilon^+) \setminus W_\epsilon^-   \label{eq:<+label+>}
\end{equation}
is a relatively compact open neighbourhood of $\sY (S) \setminus W_\epsilon^-$ in $\sX(S)$ because $f$
contracts $Y$ to $p_+$ and $W_\epsilon^-$ is $f^{-1}$-invariant. The complement of $f^{-1} (W_\epsilon^+) \setminus
W_\epsilon^-$ in $\sX (S)$ is a compact subset of $\A^N (\C) \cap q^{-1} (S)$.

Now, since the sequence of integrals $\int_{B(\C)} \log^+ \left| \left| f^n (p(b)) \right|  \right| d \mu_\C(b)$ is
unbounded, there exists a sequence $b_n \in \supp \mu_\C$ and a strictly increasing sequence of positive integers $T_n$ such that
\begin{equation}
  \log^+ \left| \left| f^{T_n} (p(b_n)) \right| \right| \geq n.
  \label{eq:size-explodes}
\end{equation}
Indeed, let $H$ be the total mass of $\mu_\C$, if we pick $T_n$ such that $\int_{B(\C)} \log^+ \left| \left| f^{T_n}
(p(b)) \right|  \right| d \mu_\C (b) \geq 2Hn$, then the set of $b$ such that  $\log^+ \left| \left| f^{T_n} (p(b))
\right|  \right| \geq n$ must be of positive measure. The sequence $f^{T_n}(p(b_n))$ cannot intersect the set
$W_\epsilon^-$, otherwise for some $n$ we would have $f^{T_n} (p(b_n)) \in W_\epsilon^-$ and by the
$f^{-1}$-invariance of $W_\epsilon^-$ we would get $p (b_n) \in W_\epsilon^-$ which contradicts
\eqref{eq:condition-intersection}. Now, we must have for $n$ large enough that
\begin{equation}
  f^{T_n} (p(b_n)) \in f^{-1} (W_\epsilon^+) \setminus W_\epsilon^-
  \label{eq:<+label+>}
\end{equation}
because otherwise the sequence $f^{T_n} (x(b_n))$ would be contained in a compact subset of $\A^N(\C)$ which
would contradict \eqref{eq:size-explodes}. Fix $b = b_n$ such a $b_n$, we have
\begin{equation}
  f^n (p(b)) \xrightarrow[n \rightarrow +\infty]{} p_+ (b)
  \label{eq:<+label+>}
\end{equation}
and the absolute value $\left| \cdot \right|_b$ is the desired absolute value.

\section{The algebraic torus}\label{sec:algebraic-torus}
\subsection{The group $\Aut (\G_m^2)$}\label{subsec:aut-group-algebraic-torus}
Any $K$-automorphism of $\Torus$ is of the form
\begin{equation}
  f (x,y) = (\alpha x^a y^b, \beta x^c y^d)
  \label{eq:<+label+>}
\end{equation}
where $\alpha, \beta \in K^\times$ and $ M_f = \begin{pmatrix}
  a & b \\
  c & d
\end{pmatrix}
\in \GL_2 (\Z)$. We will call such transformations \emph{pseudo-monomial}. The dynamical degree of $f$ is the spectral
radius of the matrix $A$. We will write the group law on $\Torus (K)$ additively and write
\begin{equation}
  f(x,y) = M_f (x,y) + b_f
  \label{eq:<+label+>}
\end{equation}
where $b_f = (\alpha, \beta)$. For pseudo-monomial transformations we have
\begin{equation}
  \lambda (f) = \rho (M_f).
  \label{eq:<+label+>}
\end{equation}

We call $M_f$ the \emph{monomial} part of $f$, we say that $M_f$ is \emph{loxodromic} if the
spectral radius $\rho(M_f)$ of $M_f$ is $> 1$ this is equivalent to the condition $\left| \Tr M_f \right| > 2 $. A
loxodromic matrix has two eigenvalues $\rho(M_f)$ and $\rho (M_f)^{-1}$. It acts on $\P^1 (\R)$ by a Möbius
transformation with exactly two irrational fixed points $v_+, v_-$. The fixed point $v_+$ is attracting with multiplier
$\frac{1}{\rho (f)}$ and $v_-$ is repulsing with with multiplier $\rho (f)$.
For two transformations $f,g$ we have
\begin{equation}
  f \circ g (x,y)= M_f M_g (x,y) + M_f (b_g) + b_f
  \label{eq:<+label+>}
\end{equation}
so that
\begin{equation}
  \Aut (\Torus) \simeq \GL_2 (\Z) \ltimes \Torus (K).
  \label{eq:<+label+>}
\end{equation}

\subsection{Dynamics at infinity}\label{subsec:dynamics-infinity}
The counterpart of Theorem \ref{thm:dynamics-loxodromic-automorphisms} is easier to establish in the case of the
algebraic torus. Start with the completion $\P^2$ of $\G_m^2$, its boundary is a triangle of lines. If we blow up any
intersection point of this triangle, we get a new completion of $\G_m^2$ with a cycle of rational curves at infinity. We
call them \emph{cyclic completions} and the intersection points of the rational curves at infinity will be called
\emph{satellite points}.

Let $X$ be a cyclic completion of $\G_m^2$ and $f \in \Aut (\G_m^2)$. We say that $f$ is \emph{algebraically stable}
over $X$ if
\begin{equation}
  \forall n \geq 0, f^n (\Ind(f^{-1})) \cap \Ind (f) = \emptyset.
  \label{eq:<+label+>}
\end{equation}
In particular, $f$ is algebraically stable if and only if $f^{-1}$ is.

\begin{thm}\label{thm:dynamics-infinity-alg-torus}
  Let $f$ be a loxodromic automorphism of $\G_m^2$, there exists a cyclic completion $X$ such that $f$  (and $f^{-1}$)
  are algebraically stable over $X$ and there is two finite disjoint sets of satellite points $\left\{ p_1, \dots, p_r
  \right\}, \left\{ q_1, \dots, q_s \right\}$ such that
  \begin{enumerate}
    \item $f$ is defined at $p_i$ and $f(p_i) = p_i$.
    \item $f^{-1}$ is defined at $q_j$ and $f^{-1}(q_j) = q_j$.
    \item For $N$ large enough, $f^N$ contracts $X \setminus \G_m^2$ to $\left\{ p_1, \dots, p_r \right\}$.
    \item For $N$ large enough, $f^{-N}$ contracts $X \setminus \G_m^2$ to $\left\{ q_1, \dots, q_s \right\}$.
    \item There exist local coordinates $(u,v)$ at $p_i$ such that $uv = 0$ is a local equation of $X \setminus \Torus$
      and
      \begin{equation}
        f (u,v) = (\alpha u^a v^b, \beta u^c v^d)
        \label{eq:<+label+>}
      \end{equation}
      where $\begin{pmatrix}
        a &b \\ c &d
      \end{pmatrix}$ is conjugated to $M_f$ by a matrix $M \in \GL_2 (\Z)$ which depends only on $p_i$.
    \item There exist local coordinates $(u,v)$ at $q_j$ such that $uv = 0$ is a local equation of $X \setminus \Torus$
      and
      \begin{equation}
        f^{-1} (u,v) = (\alpha u^a v^b, \beta u^c v^d)
        \label{eq:<+label+>}
      \end{equation}
      where $\begin{pmatrix}
        a &b \\ c &d
      \end{pmatrix}$ is conjugated to $M_{f^{-1}}$ by a matrix $M \in \GL_2 (\Z)$ which depends only on $q_i$.
\end{enumerate}
Furthermore, any cyclic completion above $X$ satisfies the same properties.
\end{thm}
Such completions will be called \emph{dynamical completions} of $f$.
\begin{proof}
  Start with the following fact. If $Y$ is a cyclic completion and $p$ is a satellite point such that $f(p) = p$ then
  there exists local coordinates at $p$ such that $f$ is pseudo-monomial monomial in these coordinates with monomial
  part conjugated $M_f$ by a matrix $M$ that depends only on $p$. Indeed, let $\pi : Y \rightarrow
  \P^2$ be the composition of blow-ups. Let $[X:Y:Z]$ be the
  projective coordinates over $\P^2$ and suppose for example that $\pi (p) = [1:0:0]$. Write $f (x,y) = M_f (x,y) + b$,
  then in the affine coordinates $(u,v) = (Y/X, Z/X)$ over the affine open subset $\left\{ X \neq 0 \right\}$, $f$
  induces a pseudo-monomial rational map with monomial part equal to $\begin{pmatrix}
    -1 & 1 \\ -1 & 0
  \end{pmatrix}
  M_f \begin{pmatrix}
    -1 & 1 \\ -1 & 0
  \end{pmatrix}^{-1}$. Now, $\pi$ is a composition of blow-up of satellite points. Let $\tau : Z \rightarrow X$ be the
  blow-up of a satellite point where $X$ is a cyclic completion. Let $p$ be one of the two satellite points belonging to the
  exceptional divisor. There exists local coordinates $(z,w)$ at $p$ and $u,v$ at $\tau (p)$ such that $zw = 0$ is a
  local equation of $Z \setminus \Torus$ and $uv = 0$ is a local equation of $X \setminus \Torus$ and such that
  \begin{equation}
    \tau (z,w) = (zw, w) \text{ or } \tau (z,w) = (z, zw).
    \label{eq:<+label+>}
  \end{equation}
  which corresponds respectively to the matrix $M_1 = \begin{pmatrix}
    1 & 1 \\
    0 & 1
  \end{pmatrix}
  $ and $M_2 = \begin{pmatrix}
    1 & 0 \\
    1 & 1
  \end{pmatrix}
  $ therefore there exists local coordinates $(u,v)$ at $p \in X$ such that $uv = 0$ is a local equation of $X \setminus
  \Torus$ at $p$ and $\pi^{-1} \circ f \circ \pi (u,v)$ is pseudo-monomial
  with monomial part of the form $M M_f M^{-1}$ where $M$ is a product of $M_1$ and $M_2$ that depends only on $p$.

  Now, for any cyclic completion $Y$, the indeterminacy points of $f^{\pm 1}$ can only be satellite points because of a
  combinatorial argument (see for example, \cite{cantatCommensuratingActionsBirational2019} Lemma 8.3), this also
  implies that if $f$ is defined at a satellite point $p$, then $f(p)$ must also be a satellite point. From
  \cite{dillerDynamicsBimeromorphicMaps2001}, we know that up to blowing up indeterminacy points of $f^{\pm 1}$ we will
  end up with an algebraically stable model of $f$. Putting this two fact together we get that there exists a cyclic
  completion $X$ such that $f$ and $f^{-1}$ are algebraically stable. Now, take $E$ an irreducible component of $X
  \setminus \G_m^2$, we show that for $N$
  large enough $f^N (E)$ must be contracted (to a satellite point). Otherwise, there would exist $N_0$ such that
  $f^{N_0} (E) = E$ and up to replacing $N_0$ by $2N_0$ we must have that the two satellite points of $E$ are fixed by
  $f^{N_0}_{|E}$.
  Let $p$ be one of them, either $f^{N_0}$ or $f^{-N_0}$ must be defined at $p$ by algebraic stability. Suppose that
  $f^{N_0}$ is, then in local coordinates $(u,v)$ at $p$ where $u = 0$ is a local equation of $E$ and $v=0$ is the other
  irreducible curve $F$ in $X \setminus X_0$ such that $p = E \cap F$ we have
  \begin{equation}
    f^{N_0} (u,v) = (\alpha u^a v^b, \beta v^d)    \label{eq:<+label+>}
  \end{equation}
  where $a,b,d \geq 0$ and the matrix $\begin{pmatrix}
    a &b \\ 0 & d
  \end{pmatrix}$ is conjugated in $\GL_2 (\Z)$ to the matrix $M_f$. This implies that $a = d = 1$ and $\Tr M_f = 2$
  then $A$ is not loxodromic, this is absurd.

  Now if $E$ is contracted to a satellite point $p$ by $f^N$, then $p$ is an indeterminacy point of $f^{-N}$ and thus
  cannot be an indeterminacy point of $f$ by algebraic stability. Thus, the forward orbit of $E$ is well defined and
  ends up consisting only of satellite points. Since, there are only finitely many of them, the forward orbit of $E$ must stop
  at a satellite point $p$ which is a fixed point of $f$. We define $p_1, \dots, p_r$ for the finite set of fixed
  satellite points that appear when doing this algorithm with every irreducible component $E$. And we define $q_1
  ,\dots, q_s$ for the satellite points defined by this algorithm with $f^{-1}$ instead of $f$. They satisfy the theorem.
\end{proof}

\begin{rmq}\label{rmq:fixed-point-is-center-eigenvaluation}
Each $p_i$ correspond to an eigenvaluation of $f$. Indeed, if $f$ is monomial at $p_i$ with a matrix $A = \begin{pmatrix}
  a & b \\ c &d
\end{pmatrix}
$, then for any monomial valuation $v_{s,t}$ at $p_i$, we have $f_* v_{s,t} = v_{as + tb, cs +t d}$ and $f_* v_* =
\lambda v_*$ for $v_* = v_{s,t}$ with $s,t$ an eigenvector of the matrix $A$ for the eigenvalue $\lambda$, the
ratio $s/t$ must be irrational. Since $A$ is a loxodromic matrix, the fixed point $v_*$ is attracting and for any
monomial valuation $v$ at $p_i$ we have
$\frac{1}{\lambda^n} f^n_* v \rightarrow v_*$. if we blow up $p_i$, then the valuation $v_*$ will become a monomial
valuation at a satellite point above $p_i$ and every divisor that was contracted to
$p_i$ will be contracted to the center of $v_*$ on this new model.
\end{rmq}

\begin{cor}\label{cor:no-invariant-curves-algebraic-torus}
  A loxodromic automorphism of $\G_m^2$ cannot admit an invariant curve.
\end{cor}
\begin{proof}
  Let $f$ be a loxodromic automorphism and $C$ be an invariant curve. We fix a completion $X$ of $\Torus$ that satisfies
  Theorem \ref{thm:dynamics-infinity-alg-torus}. Let $\overline C$ be the Zariski closure of $C$ in $X$, the curve
  $\overline C$ must intersect $X \setminus \Torus$ and the intersection points must be one the $p_i$'s or one of the
  $q_j$'s. Suppose $p_1 \in \overline C$, then $f_{|\overline C}$ is an automorphism of a curve and $p_1$ is a fixed
  point of $f_{|\overline C}$ but by the monomial local normal form of $f$ at $p$ we have that the differential of
  $f_{|\overline C}$ at $p$ is zero and this is a contradiction.
\end{proof}

\begin{prop}\label{prop:convergence-infini-alg-torus}
  Let $K$ be a complete field with an absolute value $\left| \cdot \right|$ and $f$ a loxodromic automorphism of $\G_m^2$ defined
  over $K$. If $p \in \Torus(K)$ is such that the forward $f$-orbit of $p$ is unbounded, then for any dynamical
  completion $X$ of $f$, there exists $i_0$ such that $f^n (p) \xrightarrow[n \rightarrow +\infty]{} p_{i_0}$.
\end{prop}
\begin{proof}
  Since $X(K)$ is compact, the sequence $f^n (p)$ must accumulate to a point $q \in X(K) \setminus \Torus (K)$. If $q$
  is one of the $p_i$, then because $p_i$ is a local attracting fixed point of $f$ we must have that $f^n (p)
  \rightarrow p_i$.

  If $q = q_j$, then since $q_j$ is a local attracting fixed point of $f^{-1}$ we must have that $q$ belongs to any small enough
  Euclidian neighbourhood of $q_j$ because any such neighbourhood is $f^{-1}$-invariant and this is a contradiction.

  Finally, if $q$ is any other point at infinity, then for some $N_0$ large enough, we have $f^{N_0} (q) = p_i$ for some
  $i$ and by continuity we fall back to the case $q = p_i$.
\end{proof}

\subsection{Proof of the Theorem}\label{subsec:proof-thm-alg-torus}
Let $f,g$ be two loxodromic automorphisms of $\G_m^2$ and suppose that there exists $p,q \in \G_m^2 (K)$ such that
$\OO_f (p) \cap \OO_g (q)$ is infinite. By Lemma \ref{lemme:reductions} \ref{item:same-point} we can suppose that
$p = q$ and by conjugation with the translation $(x,y) \mapsto (x,y) + p$ we can suppose that $p = (1,1)$.

Let $\left| \cdot \right|_v$ be an absolute value over $K$. Let $M_f$ be the monomial part of $f$ and $b_f = (\alpha,
\beta)$ be the translation part of $f$, we define the
notation $\log \left| b_f \right|_v := (\log \left| \alpha \right|_v, \log \left| \beta \right|_v)$. Write $f^n (p) = (\alpha_n,
\beta_n)$ and define $u_n = (\log \left| \alpha_n \right|_v, \log \left| \beta_n \right|_v)$, then $u_n$ satisfies
\begin{equation}
  u_{n+1} = M_f u_n + \log \left| b_f \right|_v.
  \label{eq:<+label+>}
\end{equation}
The matrix $M_f$ has eigenvalues $\lambda (f)$ and $1/\lambda (f)$ with eigenvectors $w_+$ and $w_-$, thus $u_n$ is of
the form
\begin{equation}
  u_n = a_+(v) \lambda (f)^n w_+ + a_- (v) \frac{1}{\lambda (f)^n} w_- - w_0(v)
  \label{eq:asymptotique-suite}
\end{equation}
where $w_0(v) = (\id - M_f)^{-1} \log \left| b_f \right|_v = a_+(v) w_+ + a_- (v) w_-$. Notice that $(\id - M_f)$ is
indeed an invertible matrix because $\lambda (f) \neq 1$.

\begin{lemme}\label{lemme:place-orbit-unbounded-torus}
There exists an absolute value $\left| \cdot \right|$ over $K$ such the sequence $(f^n (p))$ is unbounded in $\G_m^2
(K)$.
\end{lemme}
\begin{proof}
  For any absolute value $v$, by \eqref{eq:asymptotique-suite} the sequence $(f^n (p))_{n \geq 0}$ is bounded with
  respect to $v$ if and only if $a_+ (v) = a_- (v) = 0$.
   If that was the case for every absolute value $\left| \cdot
\right|$, then we would get for any absolute value $\left| \left| f^n (p) \right| \right| = \left| \left| p
\right| \right|$ and therefore $h(f^n(p)) = h(p)$ for any height function $h$. By the Northcott property, $p$ would
be $f$-periodic.
\end{proof}

\begin{prop}\label{prop:same-matrix}
  If $f,g$ are loxodromic automorphisms of $\G_m^2$ such that $\OO_f (p) \cap \OO_g (q)$ is infinite, then there exists
  $m,n \in \Z \setminus \left\{ 0 \right\}$ such that $M_f^n = M_g^m$.
\end{prop}
\begin{proof}
  We can suppose that $p = q = (1,1)$ and $ \OO_{f,+} (p) \cap \OO_{g,+} (p)$ is infinite.  By Lemma
  \ref{lemme:place-orbit-unbounded-torus}, there exists an absolute
  value $\left| \cdot \right|$ such that the $f,g$-forward orbit of $p$ is unbounded. Let $X$ be a cyclic dynamical completion
  of $f$ and $g$. By Proposition
  \ref{prop:convergence-infini-alg-torus}, $f^n (p)$ and $g^n (p)$ must converge towards the same satellite point
  $p_{i_0}$ at infinity and this must be true for any cyclic completion above $X$. Therefore, by Remark
  \ref{rmq:fixed-point-is-center-eigenvaluation}, $f$ and $g$ have the same eigenvaluation $v_*$ at $p_{i_0}$.
  Therefore, for every
  $h \in \langle f,g \rangle$, there exists $t_h$ such that $h_* v_* = t_h v_*$ and $t_h$ or $t_h^{-1}$ must be the
  dynamical degree of $h$.  Now applying Remark \ref{rmq:dynamical-degrees-discrete-subgroup} with the map $h \in
  \langle f, g \rangle \mapsto \log t_h $ we have that $\lambda(f)^n = \lambda(g)^m$ for some $n,m \in \Z \setminus
  \left\{ 0 \right\}$. This implies that the monomial form $A$ of $f^n g^{-m}$ at $p_{i_0}$ acting on $\P^1 (\R)$ has an
  irrational fixed point $t_*$ and such that the derivative satisfies $A' (t_*) =1$. Since $a \in \GL_2 (\Z)$ this
  implies that $A$ is the identity matrix. Since the matrix of the
  monomial form of $f,g$ at $p_{i_0}$ is equal to $M M_f M^{-1}$ and $M M_g M^{-1}$ respectively for some matrix $M$ we
  must have $M_f^m = M_g^n$.
\end{proof}

We can now finish the proof.
\paragraph{\textbf{Proof of Theorem \ref{thm:relation-orbits-any-char} for $\Torus$}}
Suppose $f,g$ are loxodromic automorphisms of $\G_m^2$ defined over a field $K$ and such that there exists $p,q$ such
that $\OO_f (p) \cap \OO_g (q)$ is infinite. Then, up to taking iterates we can suppose that $M_f = M_g$ by Proposition
\ref{prop:same-matrix} and Lemma \ref{lemme:reductions} \ref{item:iterate}. In particular, we have $\lambda (f) =
\lambda (g) =: \lambda$.

By Lemma \ref{lemme:reductions} \ref{item:same-point}, we can suppose that $p = q = (1,1)$ and that $\OO_{f, +} (p) \cap
\OO_{g,+} (p)$ is
infinite. By Lemma \ref{lemme:place-orbit-unbounded-torus} we can suppose that the forward $f,g$-orbit of $p =
(1,1)$ is unbounded for some fixed absolute value $\left| \cdot \right|_v$. Now using Equation
\eqref{eq:asymptotique-suite} we have
\begin{align}
  u_n (f) &= a_+^f (v) \lambda^n w_+ + a_-^f (v) \frac{1}{\lambda^n} w_- - w_0 (f) \\
  u_n (g) &= a_+^g (v) \lambda^n w_+ + a_-^g (v) \frac{1}{\lambda^n} w_- - w_0 (g)
  \label{<+label+>}
\end{align}
And $a_+^f(v) a_-^f (v) \neq 0$.
In particular, there exists a positive integer $C > 0$ such that for $n,m \geq 0$ large enough $u_n (f) = u_m (g)$ implies that
$0 = \left| m-n \right| \leq C$. Therefore there exists $l \in \left\{ -C, \dots, C  \right\}$ such that for infinitely
many $n \geq 0$ we have
\begin{equation}
  f^n (p) = g^{n+l} (p).
  \label{eq:<+label+>}
\end{equation}

Write $(\alpha_f, \beta_f) \in \Torus(K)$ for the translation part of $f$ and define
$(\alpha_g, \beta_g)$ similarly. Let $G$ be the subgroup of $\Torus(K)$ generated by
\begin{equation}
  (\alpha_h, 1), (\beta_h, 1), (1, \alpha_h), (1, \beta_h)
  \label{eq:<+label+>}
\end{equation}
for $h = f,g$. The automorphism $f,g$ restrict to selfmaps $f,g : G \rightarrow G$ of the form $\phi(x) = Ax +b$ where
$A : G \rightarrow G$ is a group homomorphism. Let $H$ be the subgroup of $G^2$ defined by
\begin{equation}
  H = \left\{ (u,v) \in G^2 : u = A^l (v) \right\}.
  \label{eq:<+label+>}
\end{equation}
Then, we have
\begin{equation}
  V = \left\{ (x,y) \in G^2 : x = g^l (y) \right\} = H - (0, A^{-l} b_l )
  \label{eq:<+label+>}
\end{equation}
where $g^l (x) = A^l x + b_l$. Since $A^2 - (\Tr A) A + (\det A) \id = 0$ we have by Theorem 4.1 of
\cite{ghiocaDynamicalMordellLangConjecture} that the set
\begin{equation}
  \left\{ n \geq 0: (f^n (1,1), g^n(1,1)) \in V \right\}
  \label{eq:<+label+>}
\end{equation}
is a finite union of arithmetic progression.

Thus, there exists $a,b \in \Z$ such that for every $k \geq 0$,
\begin{equation}
  f^{ak+b} (p) = g^{ak+b+l} (p).
  \label{eq:<+label+>}
\end{equation}
So by setting $x = f^b (p)$ and $y = g^{b+l} (q)$ and replacing $f,g$ by $f^a, g^a$, we have for every $k \geq 0$
\begin{equation}
  f^k (x) = g^k (y).
  \label{eq:<+label+>}
\end{equation}
Thus, on the set $\OO_{f, +} (x)$ we have $f = g$. The Zariski closure of the forward $f$-orbit of $x$ is Zariski dense
because a loxodromic automorphism cannot have an invariant curve by Corollary
\ref{cor:no-invariant-curves-algebraic-torus} and the result is shown.

\section{The group $\Aut_F(\A^2_K)$}\label{sec:aut-F-A2}
We prove Theorem \ref{thm:relation-orbits-any-char} for $f,g \in \GL_2 (A) \ltimes (\G_a (K) \times \G_a (K))$ with the
additional hypothesis on the positivity of the Banach density of the set
\begin{equation}
  \left\{ n \in \Z : \exists m \in \Z,  f^n (p) = g^m (q) \right\}
  \label{eq:<+label+>}
\end{equation}
By Lemma \ref{lemme:reductions} \ref{item:plus} and \ref{item:same-point}, we can suppose that $p=q$ and that the set
\begin{equation}
  \left\{ n \in \Z_{\geq 0} : \exists m \in \Z_{\geq 0},  f^n (p) = g^m (q) \right\}
  \label{eq:<+label+>}
\end{equation}
is of positive Banach density. Now, by
Lemma \ref{lemme:place-orbit-unbounded} and Corollary \ref{cor:local} we have that $f,g$ have the same eigenvaluation
$v_+$ and that up to replacing $f,g$ by some iterates we have $\lambda(f) = \lambda (g)$ by Remark
\ref{rmq:dynamical-degrees-discrete-subgroup}. In the case of polynomial
automorphism of the plane, the dynamical degree is an integer $d$. Let $X$ be a dynamical completion of $f$ and $g$, by
\cite{abboudDynamicsEndomorphismsAffine2023} Theorem 14.4, the local normal form at $p_+$ of $f$ and $g$ is of the form
\begin{equation}
  f (z,w) = (z^d \phi_1 (z,w), \phi_2 (z,w)), \quad g (z,w) = (z^d \psi_1 (z,w), \psi_2 (z,w)).
  \label{eq:<+label+>}
\end{equation}
where $\phi_1, \psi_1$ are regular invertible functions near $p_+$. We pick an absolute value $\left| \cdot \right|$
over $K$ such that $f^n (p) \rightarrow p_+$. Then, since $\phi_1, \psi_1$ are bounded non-vanishing continuous
functions on a small compact neighbourhood of $p_+$ in $X (K)$, looking at the first coordinate we have that for $n \geq
0$ large enough there exists $C > 0$ such that
\begin{equation}
  f^n (p) = g^m (p) \Rightarrow \left| n - m \right| \leq C.
  \label{eq:<+label+>}
\end{equation}
Thus by a similar argument as in the $\Torus$-case there exists $j_0 \in \left\{ -C, \dots, C \right\}$ such that the set
\begin{equation}
  \left\{ n \in \Z_{\geq 0} : f^n (p) = g^{n+j_0} (p) \right\} = \left\{ n \in \Z_{\geq 0} : (f,g)^n (p, g^{j_0}
  (p)) \in \Delta \right\}
  \label{eq:<+label+>}
\end{equation}
is of positive Banach density where $\Delta$ is the diagonal in $\A^2_K \times \A^2_K$. By Proposition 1.6 of
\cite{bellDynamicalMordellLangProblem2015}, this set contains an arithmetic progression and we conclude in the same way
as for the algebraic torus.

\section{Proof of Theorems \ref{thm:dyn-mordell-lang-like} and
\ref{thm:dyn-mordell-lang-like-graph}}\label{sec:proof-dynamical-mordell-lang-like}
\subsection{Proof of Theorem \ref{thm:dyn-mordell-lang-like}}\label{subsec:proof-thm-dyn-ML-like}
We can assume that $\OO_{(f,g), +} (x_0, y_0) \cap V$ is infinite.
By Theorem 1.3 of \cite{bellDynamicalMordellLangProblem2010}, there exists $a,b \in \Z_{\geq 0}$ such that for every
$n \geq 0$
\begin{equation}
  (f,g)^{an+b} (x_0, y_0) \in V.
  \label{eq:<+label+>}
\end{equation}
We replace $(x_0,y_0)$ by $(f^b (x_0), g^b (y_0))$. We show that $(f,g)^a (V) \subset V$. Let $Y$ be the closure of
$\OO_{(f,g)^a}(x_0,y_0)$, then $Y \subset V$, $Y$ is $(f,g)^a$-invariant and therefore $\dim Y \leq \dim V = 2$. If
$\dim Y = 0$, then $Y$ is a finite number of points and this is a contradiction since $Y$ is infinite.
So to show the result we need to prove that $\dim Y \neq 1$. Let $Y_i = \overline{\pi_i (Y)}$ where $\pi_1, \pi_2$ are
the two projections. Then $Y_1$ is $f^a$-invariant and $Y_2$ is $g^a$-invariant. By Corollary
\ref{cor:no-invariant-curves-algebraic-torus} and Proposition 4.19 of \cite{abboudDynamicsEndomorphismsAffine2023},
loxodromic automorphisms of normal affine surfaces cannot admit invariant curves, therefore we have two possibilities:
\begin{enumerate}
  \item $\dim Y_1 = 0$ and $\dim Y_2 = 2$ up to switching $Y_1$ and $Y_2$.
  \item $\dim Y_1 = \dim Y_2 = 2$.
\end{enumerate}
In the first case, we must have that $Y_1 = \OO_{f^a} (x_0)$ is finite and $Y = \OO_{f^a} (x_0) \times X_0$ and the result
is immediate.
In the second case, we must have $\dim Y \geq 2$ and therefore $Y = V$.

\subsection{Proof of Theorem \ref{thm:dyn-mordell-lang-like-graph}}\label{subsec:proof-thm-dyn-ML-like-graph}
Notice that if $h = \id$, then $\Gamma_h = \Delta$ is
the diagonal and Theorem \ref{thm:relation-orbits} implies that for some $n$ we have $f^n = g^n$. Now if $h \in \Aut
(X_0)$, replacing $g$ by $h \circ g \circ h^{-1}$ we get that
\begin{equation}
  \OO_f(x_0) \times \OO_{h g h^{-1}} (h(y_0)) \cap \Delta
  \label{eq:<+label+>}
\end{equation}
is infinite, so we have reduced to the case $h = \id$.
\begin{rmq}\label{rmq:positive-characteristic}
  If $K$ is of positive characteristic, then Theorem \ref{thm:dyn-mordell-lang-like-graph} also holds unless $\langle
  f,g \rangle$ is conjugated to a subgroup of $\Aut_F (\A^2_K)$ in $\Bir (\P^2)$. In that last case, the theorem would
  also hold with an additional assumption of positive density as in Theorem \ref{thm:relation-orbits}.
\end{rmq}

\bibliographystyle{alpha}
\bibliography{biblio}
\end{document}